\documentclass[12pt]{article}
\usepackage{skak}
\usepackage{amsthm}
\usepackage{amsmath,amsfonts,amssymb}
\usepackage{color,graphics}
\usepackage{indentfirst}
\usepackage{graphicx}
\usepackage{geometry}
\usepackage{titlesec}
\usepackage{caption}
\usepackage{authblk}
\usepackage[misc]{ifsym}  
\newtheorem{dingyi}{Definition}[section]
\newtheorem{yinli}[dingyi]{Lemma}
\newtheorem{dingli}[dingyi]{Theorem}
\newtheorem{tuilun}[dingyi]{Corollary}

\newtheorem{lizi}{Example}[]

\newcommand{\bol}{Bollob\'as}
\newcommand{\lov}{Lov\'asz}
\newcommand{\fu}{F\"{u}redi}
\newcommand{\ekr}{Erd\"{o}s-Ko-Rado}
\newcommand{\heg}{Heged\"{u}s}

\newcommand{\ELTE}{E\"{o}tv\"{o}s Lor\'{a}nd University}
\newcommand{\ELTEAdress}{P\'{a}zm\'{a}ny P\'{e}ter s\'{e}t\'{a}ny 1/C, Budapest, Hungary, H-1117}

\title{Some new \bol-type inequalities}
\author[$\sympawn$]{Erfei Yue}
\affil[$\sympawn$]{\footnotesize Institute of Mathematics, \ELTE,

\ELTEAdress,

\Letter\ yef9262@mail.bnu.edu.cn}
\date{}

\begin{document}
\maketitle

\begin{center}
\textbf{Abstract}
\end{center}

A family of disjoint pairs of finite sets~$\mathcal{P}=\{(A_i,B_i)\mid i\in[m]\}$ is called a~\bol~system if~$A_i\cap B_j\neq\emptyset$ for every~$i\neq j$,
and a skew~\bol~system if~$A_i\cap B_j\neq\emptyset$ for every~$i<j$.
~\bol~proved that for a~\bol~system, the inequality
\begin{equation*}
\sum_{i=1}^m\binom{|A_i|+|B_i|}{|A_i|}^{-1}\leqslant 1
\end{equation*}
holds.~\heg~and Frankl generalized this theorem to skew~\bol~systems with the inequality
\begin{equation*}
\sum_{i=1}^m\binom{|A_i|+|B_i|}{|A_i|}^{-1}\leqslant 1+n,
\end{equation*}
provided~$A_i,B_i\subseteq [n]$. In this paper, we improve this inequality to
\begin{equation*}
\sum_{i=1}^m \left((1+|A_i|+|B_i|)\binom{|A_i|+|B_i|}{|A_i|}\right)^{-1}\leqslant 1
\end{equation*}
with probabilistic method. We also generalize this result to partitions of sets on both symmetric and skew cases.

\section{Introduction}

Let~$[n]=\{1,\ldots,n\}$ be an~$n$-element set. A family~$\mathcal{F}$ consisting of some subsets of~$[n]$ is the main object of the research on extremal set theory.
The following concepts of~\bol~system and skew~\bol~system were introduced by~\bol~\cite{Sets} and Frankl~\cite{Book,Skew} respectively.

\begin{dingyi}
Suppose that~$\mathcal{P}=\{(A_i,B_i)\mid i\in[m]\}$ is a family of pairs of sets, where~$A_i,B_i\subseteq [n]$. Then~$\mathcal{P}$ is called a~\emph{\bol~system} if
\begin{equation*}
A_i\cap B_j=\emptyset\Leftrightarrow i=j;
\end{equation*}
and a \emph{skew~\bol~system} if
\begin{description}
  \item[(1)] $A_i\cap B_i=\emptyset$,
  \item[(2)] $A_i\cap B_j\neq\emptyset$ when~$i<j$.
\end{description}
\end{dingyi}

To solve a problem on hypergraphs,~\bol~\cite{Sets} proved the following theorem in 1965.
In 1975, Tarj\'an~\cite{Tarjan} gave a more elegant proof for it independently.

\begin{dingli}[\bol~\cite{Sets}]\label{Th:Bollobas}
Let~$\mathcal{P}=\{(A_i,B_i)\mid i\in[m]\}$ be a~\bol~system, where~$A_i,B_i\subseteq [n]$. Then
\begin{equation*}
\sum_{i=1}^m\frac{1}{\binom{|A_i|+|B_i|}{|A_i|}}\leqslant 1.
\end{equation*}
\end{dingli}

The following uniform version is a corollary of Theorem~\ref{Th:Bollobas} for~$|A_i|=a$ and~$|B_i|=b$,
and it was proved by Jaeger and Payan~\cite{FRA} in 1971 and Katona~\cite{Katona} in 1974 independently.

\begin{tuilun}\label{Th:uniform}
Let~$\mathcal{P}=\{(A_i,B_i)\mid i\in[m]\}$ be a~\bol~system, where~$A_i,B_i\subseteq [n]$, and~$|A_i|=a,|B_i|=b$ for any~$i$.
Then~$m\leqslant\binom{a+b}{a}$.
\end{tuilun}

Theorem~\ref{Th:Bollobas} and Corollary~\ref{Th:uniform} play a pivotal role in extremal set theory. They has spurred significant research interest,
leading to generalizations for vector spaces by~\lov~\cite{LinearSpaces,Lov1,Lov2} in 1977 and affine spaces by~\heg~\cite{AffineSpaces} in 2015.

In 1982, using a creative polynomial method, Frankl~\cite{Skew} proved that the inequality of Corollary~\ref{Th:uniform}
remains true if~$\mathcal{P}$ is a skew~\bol~system (instead of~\bol~system).
However, the condition of Theorem~\ref{Th:Bollobas} cannot be weakened in the same way.
So it is natural to ask what can we say for a (nonuniform) skew~\bol~system.
In 2023, Heged\"{u}s and Frankl~\cite{Hegedus} answered the question with the following theorem.

\begin{dingli}[Heged\"{u}s and Frankl~\cite{Hegedus}]\label{Th:Hegedus}
Suppose~$A_i,B_i\subseteq [n],i\in [m]$, and~$\mathcal{P}=\{(A_i,B_i)\mid i\in[m]\}$ is a skew~\bol~system. Then
\begin{equation*}
\sum_{i=1}^m\frac{1}{\binom{|A_i|+|B_i|}{|A_i|}}\leqslant 1+n.
\end{equation*}
\end{dingli}

Although the inequality in Theorem~\ref{Th:Hegedus} is tight, we can still strengthen it in the following way.

\begin{dingli}\label{Th:1}
Suppose~$A_i,B_i\subseteq [n],i\in [m]$, and~$\mathcal{P}=\{(A_i,B_i)\mid i\in[m]\}$ is a skew~\bol~system. Then
\begin{equation*}
\sum_{i=1}^m\frac{1}{(1+|A_i|+|B_i|)\binom{|A_i|+|B_i|}{|A_i|}}\leqslant 1.
\end{equation*}
\end{dingli}

In 1984, Furedi~\cite{Threshold} investigated the threshold (or~$t$-intersecting) case of~\bol-type theorem.
Subsequent work by Zhu~\cite{tl}, Talbot~\cite{Inequality}, and Kang, Kim, and Kim~\cite{Inequality2} further extended the results.
Recently, more variations of~\bol-type theorem are studied,
such as Kir\'{a}ly, Nagy, P\'{a}lv\"{o}lgyi, and Visontai~\cite{weakly}, O'Neill and Verstra\"{e}te~\cite{kTuples},
Scott and Wilmer~\cite{Scott}, and Yu, Kong, Xi, Zhang, and Ge~\cite{HemiBundled}.

The uniform version of~\bol~Theorem (Theorem~\ref{Th:uniform}) was generalized to~$r$-partitions by Alon~\cite{Alon}.
More specifically, the following theorem decides the maximum cardinality of a (skew)~\bol~system in which every set has a fixed size in each part of the ground set.

\begin{dingli}[Alon~\cite{Alon}]\label{Th:Sets}
Suppose~$X=[n]$ is the disjoint union of some sets~$X_1,\ldots,X_r$.
For some fixed~$a_k,b_k$, a skew~\bol~system~$\mathcal{P}=\{(A_i,B_i)\mid i\in[m]\}$ satisfies that
\begin{equation*}
A_i,B_i\subseteq X,|A_i\cap X_k|=a_k,|B_i\cap X_k|=b_k, \forall i\in[m],k\in[r].
\end{equation*}
Then we have
\begin{equation*}
m\leqslant\prod_{k=1}^r\binom{a_k+b_k}{a_k}.
\end{equation*}
\end{dingli}

In the following theorem, we generalize Theorem~\ref{Th:Sets} to the nonuniform version.

\begin{dingli}\label{Th:skew}
Suppose that~$X=[n]$ is the disjoint union of some sets~$X_1,\ldots,X_r$, and~$|X_k|=n_k$.
For some~$A_i,B_i\subseteq X$,~$\mathcal{P}=\{(A_i,B_i)\mid i\in[m]\}$ is a skew~\bol~system. Then
\begin{equation}\label{E2}
\sum_{i=1}^m\left(\prod_{k=1}^r\binom{|A_i\cap X_k|+|B_i\cap X_k|}{|A_i\cap X_k|}\right)^{-1}\leqslant\prod_{k=1}^r(1+n_k)\leqslant\left(1+\frac{n}{r}\right)^r.
\end{equation}
\end{dingli}

This theorem shows two things. Firstly, if we fixed a partition~$X=X_1\cup\cdots\cup X_r$, then the maximum value of the left hand side of~(\ref{E2}) is~$\prod_{k=1}^r(1+n_k)$.
Secondly, if we only fixed the ground set~$X$, while the partition is variable, then the maximum value of it is~$\left(1+\frac{n}{r}\right)^r$. 
The first inequality in~(\ref{E2}) is tight for all~$n_1,\ldots,n_r$, and the second inequality is also tight if~$r|n$, and~$X_k$'s are divided equally.
Note that the case~$r=1$ is Theorem~\ref{Th:Hegedus}.

We can also consider the symmetric version of Theorem~\ref{Th:skew}. Surprisingly, this time the symmetric version is harder than the skew version.
For the case~$r=2$, we can prove the following tight inequality for a~\bol~system.

\begin{dingli}\label{Th:2}
Suppose that~$X=[n]$ is the disjoint union of some sets~$X_1$ and~$X_2$,
and for some~$A_i,B_i\subseteq X$,~$\mathcal{P}=\{(A_i,B_i)\mid i\in[m]\}$ is a~\bol~system. Then
\begin{equation*}
\sum_{i=1}^m\left(\binom{|A_i\cap X_1|+|B_i\cap X_1|}{|A_i\cap X_1|}\binom{|A_i\cap X_2|+|B_i\cap X_2|}{|A_i\cap X_2|}\right)^{-1}\leqslant
1+\left\lfloor\frac{n}{2}\right\rfloor.
\end{equation*}
\end{dingli}

And for an arbitrary~$r\geqslant 1$, we can prove the following theorem.

\begin{dingli}\label{Th:sym}
Suppose that~$X=[n]$ is the disjoint union of some sets~$X_1,\ldots,X_r$,
and for some~$A_i,B_i\subseteq X$,~$\mathcal{P}=\{(A_i,B_i)\mid i\in[m]\}$ is a~\bol~system. Then
\begin{equation*}
\sum_{i=1}^m\left(\prod_{k=1}^r\binom{|A_i\cap X_k|+|B_i\cap X_k|}{|A_i\cap X_k|}\right)^{-1}\leqslant\left(1+\frac{n}{r}\right)^{r-1}.
\end{equation*}
\end{dingli}

The case~$r=1$ is the original~\bol~Theorem (Theorem~\ref{Th:Bollobas}). And the case~$r=2$ leads to an inequality
\begin{equation*}
\sum_{i=1}^m \left(\binom{|A_i\cap X_1|+|B_i\cap X_1|}{|A_i\cap X_1|}\binom{|A_i\cap X_2|+|B_i\cap X_2|}{|A_i\cap X_2|}\right)^{-1}\leqslant 1+\frac{n}{2},
\end{equation*}
which is slightly weaker than the one in Theorem~\ref{Th:2}, but still tight for even~$n$'s.
For the case~$r\geqslant 3$, unfortunately, the upper bound given in Theorem~\ref{Th:sym} does not seem to be tight.

The paper is organized as follows. Section 2 focuses on proving Theorem~\ref{Th:1} and Theorem~\ref{Th:skew}. We also discuss the tightness of Theorem~\ref{Th:skew}.
Section 3 presents proofs for Theorem~\ref{Th:2} and Theorem~\ref{Th:sym}.
Additionally, we leverage these theorems to derive a generalized form of the LYM-inequality.

\section{Proof of the skew version and its tightness}
For a fixed (finite) set~$X$, let~$S(X)$ be the group consisting of all permutations on~$X$, and~$S_n=S([n])$.
For a permutation~$\sigma\in S(X)$ and a subset~$A\subseteq X$, denote~$\sigma(A)$ the image of~$A$.
In other words,~$\sigma(A)=\{\sigma(a)\mid a\in A\}$.
In this section, we will prove Theorem~\ref{Th:1} and Theorem~\ref{Th:skew} with probabilistic method.

\begin{proof}[Proof of Theorem~\ref{Th:1}]
Pick a random permutation~$\sigma\in S_{n+1}$ uniformly. Note that
\begin{equation*}
\mathcal{P}'=\{(\sigma(A_i),\sigma(B_i))\mid (A_i,B_i)\in\mathcal{P}\}
\end{equation*}
is still a skew~\bol~system. The order of elements in~$[n+1]$ is rearranged by~$\sigma$.
For any~$i$, let~$E_i$ be the event that all elements of~$A_i$ go before~$B_i$, and they are separated by the extra element~$n+1$.
In other words,
\begin{equation*}
E_i=\{\sigma\in S_{n+1}\mid\sigma(a)<\sigma(n+1)<\sigma(b),\forall a\in A_i,b\in B_i\}.
\end{equation*}

Now we claim that~$E_i$ and~$E_j$ cannot happen at the same time if~$i\neq j$. Suppose for the contradiction that~$\sigma\in E_i\cap E_j$.
For any~$a\in A_i$ and~$b\in B_j$, we have~$\sigma(a)<\sigma(n+1)<\sigma(b)$. Hence~$\sigma(A_i)\cap\sigma(B_j)=\emptyset$.
Similarly, we can prove that~$\sigma(A_j)\cap \sigma(B_i)=\emptyset$. That is a contradiction to the fact that~$\mathcal{P}'$ is a skew~\bol~system.
Then we have
\begin{equation*}
\mathbb{P}\left(\bigcup_{i=1}^m E_i\right)=\sum_{i=1}^m\mathbb{P}(E_i).
\end{equation*}

\begin{table}[h]
\centering
\setlength{\tabcolsep}{3mm}{
\begin{tabular}{ccccccccccc}
\cline{2-4}\cline{6-6}\cline{9-11}
   & \multicolumn{3}{|c|}{$\sigma(A_i)$} & & \multicolumn{1}{|c|}{} & & & \multicolumn{3}{|c|}{$\sigma(B_i)$}  \\
\cline{2-4}\cline{6-6}\cline{9-11}
\phantom{aaa} & \phantom{aaa} & \phantom{aaa} & \phantom{aaa} & \multicolumn{3}{c}{$\sigma(n+1)$} & \phantom{aaa} & \phantom{aaa} & \phantom{aaa} & \phantom{aaa} \\
\cline{1-3}\cline{6-6}\cline{8-10}
\multicolumn{3}{|c|}{$\sigma(A_j)$} & & & \multicolumn{1}{|c|}{} & & \multicolumn{3}{|c|}{$\sigma(B_j)$} &   \\
\cline{1-3}\cline{6-6}\cline{8-10}
\end{tabular}}
\caption*{A schematic diagram for the proof}
\end{table}

To finish the proof, we need to calculate the probability of each~$E_i$.
Note that we only need to consider the order of elements in~$\sigma(A_i\cup B_i\cup\{n+1\})$.
Within the~$1+|A_i|+|B_i|$ places, there is exactly~$1$ slot for~$n+1$.
And for the elements in~$A_i\cup B_i$, the chance of any~$|A_i|$ elements that lead the series is equal.
There are~$\binom{|A_i|+|B_i|}{|A_i|}$ different possibilities, and exactly~$1$ of them is what we need.
Hence~$\mathbb{P}(E_i)=\left((1+|A_i|+|B_i|)\binom{|A_i|+|B_i|}{|A_i|}\right)^{-1}$, and so
\begin{equation*}
1\geqslant\mathbb{P}\left(\bigcup_{i=1}^m E_i\right)=\sum_{i=1}^m\frac{1}{(1+|A_i|+|B_i|)\binom{|A_i|+|B_i|}{|A_i|}}.
\end{equation*}
\end{proof}

Using a similar idea, we can also prove Theorem~\ref{Th:skew}, but the progress would be more tricky.
For a finite set~$X=\{x_1,\ldots,x_k\}$, where~$x_1<\cdots<x_k$, let~$X^+=X\cup\{x_k+\frac{1}{2}\}$. 

\begin{proof}[Proof of Theorem~\ref{Th:skew}]
Without loss of generality, suppose~$X_1=\{1,\ldots,n_1\},X_2=\{n_1+1,\ldots,n_1+n_2\},\ldots,X_r=\{n_1+\cdots+n_{r-1}+1,\ldots,n_1+\cdots+n_r\}$,
where~$n_1+\cdots+n_r=n$. Let~$X'=X_1^+\cup\cdots\cup X_r^+$, and
\begin{equation*}
\Omega=\{\sigma\in S(X')\mid \sigma(X_k^+)=X_k^+,\forall k\in[r]\}.
\end{equation*}
Pick a random permutation~$\sigma\in\Omega$ uniformly. Then
\begin{equation*}
\mathcal{P}'=\{(\sigma(A_i),\sigma(B_i))\mid (A_i,B_i)\in\mathcal{P}\}
\end{equation*}
is still a skew~\bol~system. For any~$i\in[m]$ and~$k\in[r]$, let
\begin{equation*}
E_{ik}=\{\sigma\in\Omega\mid\sigma(a)<\sigma(n_1+\cdots+n_k+1/2)<\sigma(b),\forall a\in A_i\cap X_k,b\in B_i\cap X_k\},
\end{equation*}
and
\begin{equation*}
E_i=\bigcap_{k=1}^r E_{ik}.
\end{equation*}
Note that for a fixed~$i$, the events~$E_{ik}$'s are independent. Using this fact and the similar argument adopted in the proof of Theorem~\ref{Th:1}, we have
\begin{equation*}
\begin{split}
  \mathbb{P}(E_i) & =\prod_{k=1}^r\mathbb{P}(E_{ik})=\prod_{k=1}^r\left((1+|A_i\cap X_k|+|B_i\cap X_k|)\binom{|A_i\cap X_k|+|B_i\cap X_k|}{|A_i\cap X_k|}\right)^{-1} \\
    & \geqslant\prod_{k=1}^r\left((1+n_k)\binom{|A_i\cap X_k|+|B_i\cap X_k|}{|A_i\cap X_k|}\right)^{-1}.
\end{split}
\end{equation*}

Now we claim~$E_i\cap E_j=\emptyset$ for distinct~$i,j$. Suppose for the contradiction that~$\sigma\in E_i\cap E_j$.
Then for any~$k$, we have~$\sigma\in E_{ik}\cap E_{jk}$,so~$\sigma(a)<\sigma(n_1+\cdots+n_k+\frac{1}{2})<\sigma(b)$ for any~$a\in A_i\cap X_k,b\in B_j\cap X_k$.
Then~$\sigma(A_i\cap X_k)$ and~$\sigma(B_j\cap X_k)$ are disjoint. Similarly, we can prove that~$\sigma(A_j\cap X_k)$ and~$\sigma(B_i\cap X_k)$ are disjoint.
Hence~$\sigma(A_i)\cap\sigma(B_j)=\sigma(A_j)\cap\sigma(B_i)=\emptyset$, which contradicts the fact that~$\mathcal{P}'$ is a skew~\bol~system.
In this way, we have
\begin{equation*}
1\geqslant\mathbb{P}(\bigcup_{i=1}^m E_i)=\sum_{i=1}^m\mathbb{P}(E_i)
\geqslant\sum_{i=1}^m\prod_{k=1}^r\left((1+n_k)\binom{|A_i\cap X_k|+|B_i\cap X_k|}{|A_i\cap X_k|}\right)^{-1},
\end{equation*}
so
\begin{equation*}
\sum_{i=1}^m\left(\prod_{k=1}^r\binom{|A_i\cap X_k|+|B_i\cap X_k|}{|A_i\cap X_k|}\right)^{-1}
\leqslant\prod_{k=1}^r(1+n_k)\leqslant\left(\frac{1}{r}\sum_{k=1}^r(1+n_k)\right)^r=\left(1+\frac{n}{r}\right)^r.
\end{equation*}
\end{proof}

The next example shows the tightness of Theorem~\ref{Th:skew}.

\begin{lizi}
Suppose~$X=[n]$ is the disjoint union of~$X_1,\ldots,X_r$, where~$|X_k|=n_k$.
Let~$A_1,\ldots,A_{2^n}$ be all subsets of~$X$, and~$|A_1|\geqslant\cdots\geqslant|A_{2^n}|$.
Then~$\mathcal{P}=\{(A_i,[n]\setminus A_i)\mid i\in[2^n]\}$ is a skew~\bol~system.
For a fixed~$i$, let~$|A_i\cap X_k|=a_k$, and~$(a_1,\ldots,a_r)$ is called the type of~$A_i$.
Then the contribution of~$A_i$ to the sum is~$\prod_{k=1}^r\binom{n_k}{a_k}^{-1}$.
On the other hand, there are altogether~$\prod_{k=1}^r\binom{n_k}{a_k}$ many of~$A_j$'s that have the same type with~$A_i$.
So the total contribution of a type of~$A_i$'s to the sum is exactly~$1$.
Note that there are~$\prod_{k=1}^r(1+n_k)$ many of different types, because~$a_k\in\{0,1,\ldots,n_k\}$, and each~$a_k$'s are valued independently. Hence
\begin{equation*}
\sum_{i=1}^m\left(\prod_{k=1}^r\binom{|A_i\cap X_k|+|B_i\cap X_k|}{|A_i\cap X_k|}\right)^{-1}=\prod_{k=1}^r(1+n_k),
\end{equation*}
and so the first inequality of Theorem~\ref{Th:skew} is tight for any~$n$ and~$r$.
In particular, if~$r|n$, and~$n_1=\cdots=n_r=\frac{n}{r}$, the second inequality is also tight.
\end{lizi}

\section{Proof of the symmetric version and an application}

In this section, we will prove Theorem~\ref{Th:2} and Theorem~\ref{Th:sym}, and show an application of them.
Firstly, we begin this work with proving the following lemma.

\begin{yinli}\label{L1}
Suppose  that~$X=[n]$ is the disjoint union of some sets~$X_1,\ldots,X_r$, and~$|X_k|=n_k$.
For some~$A_i,B_i\subseteq X$,~$\mathcal{P}=\{(A_i,B_i)\mid i\in[m]\}$ is a~\bol~system.
Then for any~$l\in[r]$, we have
\begin{equation}\label{E1}
\sum_{i=1}^m\left(\prod_{k=1}^r\binom{|A_i\cap X_k|+|B_i\cap X_k|}{|A_i\cap X_k|}\right)^{-1}\leqslant\frac{\prod_{k=1}^r (1+n_k)}{1+n_l}.
\end{equation}
\end{yinli}
\begin{proof}
Without loss of generality, we may assume that~$X_1=\{1,\ldots,n_1\},X_2=\{n_1+1,\ldots,n_1+n_2\},\ldots,X_r=\{n_1+\cdots+n_{r-1}+1,\ldots,n_1+\cdots+n_r\}$.
For a fixed~$l$, denote
\begin{equation*}
X_k'=\begin{cases}
X_k^+, & \textrm{if~} k\neq l, \\
X_k, & \textrm{if~} k=l,
\end{cases}
\end{equation*}
and~$X'=\bigcup_{k=1}^r X_k'$. Let
\begin{equation*}
\Omega=\{\sigma\in S(X')\mid\sigma(X_k')=X_k',\forall k\in[r]\}.
\end{equation*}
Pick a random permutation~$\sigma\in\Omega$ uniformly. Then
\begin{equation*}
\mathcal{P}'=\{(\sigma(A_i),\sigma(B_i))\mid (A_i,B_i)\in\mathcal{P}\}
\end{equation*}
is still a~\bol~system. For any~$i\in[m]$ and~$k\neq l$, let
\begin{eqnarray*}
  E_{il} &=& \{\sigma\in\Omega\mid\sigma(a)<\sigma(b),\forall a\in A_i\cap X_l,b\in B_i\cap X_l\}, \\
  E_{ik} &=& \{\sigma\in\Omega\mid\sigma(a)<\sigma(n_1+\cdots+n_k+1/2)<\sigma(b),\forall a\in A_i\cap X_k,b\in B_i\cap X_k\},
\end{eqnarray*}
and
\begin{equation*}
E_i=\bigcap_{k=1}^r E_{ik}.
\end{equation*}
Note that for a fixed~$i$, the events~$E_{ik}$'s are independent, and we have
\begin{equation*}
\mathbb{P}(E_{il})=\binom{|A_i\cap X_k|+|B_i\cap X_k|}{|A_i\cap X_k|}^{-1},
\end{equation*}
and
\begin{equation*}
\begin{split}
  \mathbb{P}(E_{ik}) & =\left((1+|A_i\cap X_k|+|B_i\cap X_k|)\binom{|A_i\cap X_k|+|B_i\cap X_k|}{|A_i\cap X_k|}\right)^{-1} \\
    & \geqslant\left((1+n_k)\binom{|A_i\cap X_k|+|B_i\cap X_k|}{|A_i\cap X_k|}\right)^{-1},
\end{split}
\end{equation*}
for any~$k\neq l$, so
\begin{equation*}
\mathbb{P}(E_i)\geqslant(n_l+1)\prod_{k=1}^r\left((n_k+1)\binom{|A_i\cap X_k|+|B_i\cap X_k|}{|A_i\cap X_k|}\right)^{-1}.
\end{equation*}

Now we claim that~$E_i\cap E_j=\emptyset$ for distinct~$i,j$. Suppose on the contrary there is a~$\sigma\in E_i\cap E_j$.
Then like the discussion in the proof of Theorem~\ref{Th:skew},
we can prove that~$\sigma(A_i\cap X_k)$ and~$\sigma(B_j\cap X_k)$ are disjoint for any~$k\neq l$, as well as~$\sigma(A_j\cap X_k)$ and~$\sigma(B_i\cap X_k)$.
Suppose~$x\in\sigma(A_i\cap X_l)\cap\sigma(B_j\cap X_l)$. Then we have~$\sigma(a)<x<\sigma(b)$ for any~$a\in A_j\cap X_l$ and~$b\in B_i\cap X_l$,
so~$\sigma(A_j\cap X_l)$ and~$\sigma(B_i\cap X_l)$ are disjoint, which contradicts the fact that~$\mathcal{P}'$ is a~\bol~system.

\begin{table}[h]
\centering
\setlength{\tabcolsep}{3mm}{
\begin{tabular}{ccccccccc}
\cline{2-5}\cline{7-9}
   & \multicolumn{4}{|c|}{$\sigma(A_i\cap X_l)$} & & \multicolumn{3}{|c|}{$\sigma(B_i\cap X_l)$}  \\
\cline{2-5}\cline{7-9}
\phantom{aaa} & \phantom{aaa} & \phantom{aaa} & \phantom{a} & $x$ & \phantom{a} & \phantom{aaa} & \phantom{aaa} & \phantom{aaa} \\
\cline{1-3}\cline{5-8}
\multicolumn{3}{|c|}{$\sigma(A_j\cap X_l)$} & & \multicolumn{4}{|c|}{$\sigma(B_j\cap X_l)$} &   \\
\cline{1-3}\cline{5-8}
\end{tabular}}
\caption*{A schematic diagram for the proof}
\end{table}

Then we have
\begin{equation*}
1\geqslant\mathbb{P}(\bigcup_{i=1}^m E_i)\geqslant\sum_{i=1}^m (1+n_l)\prod_{k=1}^r\left((1+n_k)\binom{|A_i\cap X_k|+|B_i\cap X_k|}{|A_i\cap X_k|}\right)^{-1},
\end{equation*}
and so
\begin{equation*}
\sum_{i=1}^m\prod_{k=1}^r\binom{|A_i\cap X_k|+|B_i\cap X_k|}{|A_i\cap X_k|}^{-1}\leqslant\frac{\prod_{k=1}^r(1+n_k)}{1+n_l}.
\end{equation*}
\end{proof}

For the case~$r=2$, Lemma~\ref{L1} implies
\begin{equation*}
\sum_{i=1}^m\ \left(\binom{|A_i\cap X_1|+|B_i\cap X_1|}{|A_i\cap X_1|}\binom{|A_i\cap X_2|+|B_i\cap X_2|}{|A_i\cap X_2|}\right)^{-1}
\leqslant \min\{1+n_1,1+n_2\}\leqslant 1+\left\lfloor\frac{n}{2}\right\rfloor,
\end{equation*}
which is the result of Theorem~\ref{Th:2}. The following example proves the tightness of it.

\begin{lizi}
Suppose~$k=\lfloor\frac{n}{2}\rfloor,m=\binom{n}{k}$, and~$X=[n],X_1=[k],X_2=X\setminus X_1$.
Let~$\binom{X}{k}=\{A_1,\ldots,A_m\}$, and~$B_i=[n]\setminus A_i$. Then~$\mathcal{P}=\{(A_i,B_i)\mid i\in[m]\}$ is a~\bol~system, and
\begin{equation*}
\sum_{i=1}^m\left(\binom{|A_i\cap X_1|+|B_i\cap X_1|}{|A_i\cap X_1|}\binom{|A_i\cap X_2|+|B_i\cap X_2|}{|A_i\cap X_2|}\right)^{-1}=1+k.
\end{equation*}
\end{lizi}

For the general case of an arbitrary~$r$, taking the product of all~$l\in[r]$ on inequality (\ref{E1}), we have
\begin{equation*}
\sum_{i=1}^m\left(\prod_{k=1}^r\binom{|A_i\cap X_k|+|B_i\cap X_k|}{|A_i\cap X_k|}\right)^{-1}
\leqslant\left(\prod_{k=1}^r(1+n_k)\right)^{\frac{r-1}{r}}\leqslant\left(1+\frac{n}{r}\right)^{r-1},
\end{equation*}
which is the statement of Theorem~\ref{Th:sym}.

A family~$\mathcal{F}$ of subsets is called an anti-chain, if~$A\not\subseteq B$ for any~$A,B\in\mathcal{F}$.
The following LYM-inequality on anti-chains is introduced by Yamamoto~\cite{Y}, Meshalkin~\cite{M}, and Lubell~\cite{L},
and plays an important role in extremal set theory.

\begin{dingli}[LYM-inequality~\cite{L,M,Y}]
Suppose that~$F_1,\ldots,F_m\subseteq [n]$, and the family~$\mathcal{F}=\{F_1,\ldots,F_m\}$ is an anti-chain. Then we have
\begin{equation*}
\sum_{i=1}^m \binom{n}{|F_i|}^{-1}\leqslant 1.
\end{equation*}
\end{dingli}

Let~$A_i=F_i,B_i=[n]\setminus F_i$, and~$\mathcal{P}=\{(A_i,B_i)\mid i\in[m]\}$.
Then~$\mathcal{F}$ is an anti-chain implies~$\mathcal{P}$ is a~\bol~system.
So Theorem~\ref{Th:2} and Theorem~\ref{Th:sym} lead to the following two variations of LYM-inequality.

\begin{tuilun}
Suppose that~$X=[n]$ is the disjoint union of~$X_1$ and~$X_2$,
and for some~$F_1,\ldots,F_m\subseteq [n]$, the family~$\mathcal{F}=\{F_1,\ldots,F_m\}$ is an anti-chain. Then we have
\begin{equation*}
\sum_{i=1}^m \left(\binom{|X_1|}{|F_i\cap X_1|}\binom{|X_2|}{|F_i\cap X_2|}\right)^{-1}\leqslant 1+\left\lfloor\frac{n}{2}\right\rfloor.
\end{equation*}
\end{tuilun}

\begin{tuilun}
Suppose that~$X=[n]$ is the disjoint union of~$X_1,\ldots,X_r$,
and for some~$F_1,\ldots,F_m\subseteq [n]$, the family~$\mathcal{F}=\{F_1,\ldots,F_m\}$ is an anti-chain. Then we have
\begin{equation*}
\sum_{i=1}^m \left(\prod_{k=1}^r\binom{|X_k|}{|F_i\cap X_k|}\right)^{-1}\leqslant \left(1+\frac{n}{r}\right)^{r-1}.
\end{equation*}
\end{tuilun}

\bibliography{bib}

\end{document}